\documentclass[11pt]{article}%

\usepackage{capt-of}
\usepackage{amssymb}
\usepackage{amsfonts}
\usepackage{amsmath} 
\usepackage[nohead]{geometry}
\usepackage[singlespacing]{setspace}
\usepackage[bottom]{footmisc}%
\usepackage{indentfirst}
\usepackage{graphicx}%
\usepackage{subcaption}
\usepackage{booktabs}
\usepackage{float}
\usepackage{mathrsfs}
\usepackage{authblk}
\usepackage{multirow}
\usepackage[usenames, dvipsnames]{color}
\usepackage{bbm}
\usepackage[english]{babel}
\usepackage[normalem]{ulem}

\usepackage{natbib}
\setcitestyle{authoryear,open={(},close={)}}

\newtheorem{definition}{Definition}[section]

\newtheorem{lemma}{Lemma}[section]

\newtheorem{remark}{Remark}[section]

\newtheorem{theorem}{Theorem}[section]

\newenvironment{proof}[1][Proof]{\noindent\textbf{#1.} }{\ \rule{0.5em}{0.5em}}

\newcommand{\M}{\mathcal{M}}

\newcommand{\sumIN}{\sum_{i=0}^{N-1}}
\newcommand{\bOne}{\textbf{1}}

\numberwithin{equation}{section} \hfuzz3pt

\begin{document}

\title{Portfolio Optimization under Correlation Constraint}
\author[1]{Aditya Maheshwari\thanks{Email: aditya.maheshwari0210@gmail.com. The views expressed in this article are the authors and do not necessarily reflect the views of Bank of America Securities.}}
\author[2]{Traian Pirvu\thanks{Email: tpirvu@math.mcmaster.ca}}
\affil[1]{Bank of America Securities, New York}
\affil[2]{Department of Mathematics and Statistics , McMaster University, Canada}

\maketitle

\sloppy


\begin{abstract}
We consider the problem of portfolio optimization with a correlation constraint.
The framework is the multiperiod stochastic financial market setting with one tradable stock, stochastic income and a non-tradable index.
The correlation constraint is imposed on the portfolio and the non-tradable index at some benchmark time horizon. The goal is to maximize portofolio's expected exponential utility subject to the correlation constraint. Two types of optimal portfolio strategies are considered: the subgame perfect and the precommitment ones.   
We find analytical expressions for the constrained subgame perfect (CSGP) and the constrained precommitment (CPC) portfolio strategies. Both these portfolio strategies yield significantly lower risk when compared to the unconstrained setting, at the cost of a small utility loss. The performance of the CSGP and CPC portfolio strategies is similar.  
\end{abstract}


\textbf{Keywords: portfolio optimization, correlation constraints} 

\strut

\textbf{JEL Classification Numbers: 91B06, 60H30} 


\section{Introduction}
Mean variance optimization introduced by Markowitz \cite{HMM52} changed the pathway of how fund 
managers look at optimal investments. However, mean variance analysis is a special case of more 
general expected utility framework introduced by von Neumann and Morgenstern~\cite{NM1947}.
Our goal is to maximize expected exponential utility function subject to a correlation
constraint. Portfolio optimization under correlation constraints is of particular interest to hedge funds 
which implement market defensive strategies with higher returns during down markets than up-markets,
 and this makes our research interesting from a practical standpoint.

Although portfolio allocation within expected utility paradigm is very well explored problem in finance, few studies have  considered  portfolio allocation with correlation constraint. Recent works that stand out and are also the motivation of this study are \cite{BV14},\cite{PKP16} and \cite{BCV18}. In \cite{BV14}, the authors addressed the problem of mean  variance optimization along with correlations constraint. The work \cite{PKP16} looks at the problem of capital allocation with correlation constraint, and the objective is minimizing capital at risk.  The paper \cite{BCV18} extended \cite{BV14}, and \cite{PKP16} to law invariant preferences. All these works consider a static framework; the optimization is performed at time zero and the optimal strategies are implemented. However, if at a later time the optimization criterion and the risk constraint are dynamically updated the time zero strategy fails to remain optimal due to the correlation constraint; this is refer to as time inconsistency. In the presence of a commitment mechanism the time zero optimal strategy also known as the precommitment strategy, will be implemented. Otherwise, there will be an incentive to deviate from  the precommitment strategy and one way out of this predicament is to consider the subgame perfect strategy, also known as the Nash equilibrium strategy.      

This paper considers the expected exponential utility maximization with correlation constraint in a discrete time setting. The multi-period stochastic financial market comprises one tradable stock, stochastic income and a non-tradable index.  The non-linearity of correlation constraint turn the optimization problem time inconsistent as explained above. The constrained subgame perfect strategy (CSGP) is naturally defined in our context through backward induction. The constrained precommitment strategy (CPC) is considered as well.
 We compare the performance of these two portfolio strategies against the unconstrained subgame perfect strategy (UnSGP). 

\textbf{Contribution:} Our first and foremost contribution is the analytic solutions for the CSGP and CPC strategies. To the best of our knowledge, we are not aware of other studies which consider time inconsistency arising due to the correlation constraint in the context of portfolio optimization in a multi-period setting. Secondly, through numerical simulations we provide evidence for significant risk reduction in ``high risk'' environment by following CSGP and CPC strategies compared to the UnSGP strategy. To illustrate an example, the increase in uncertainty due to increasing the investment horizon from 2 period to 10 period results in 50\% improvement in utility at the expense of 150\% increase in risk for the UnSGP strategy. In contrast, our CSGP strategy results in 30\% improvement in utility together with 18\% \emph{reduction} in risk. Thirdly, we find the CSGP and CPC strategy to perform approximately the same for our portfolio allocation problem. This in contrast to the results in \cite{Wu2014} where authors find the CPC strategy to significantly outperform the CSGP strategy in the context of mean-variance optimization.

The rest of the paper is organized as follows - Section \ref{Model} describes the model; the stock price process, benchmark index and stochastic income stream, the trading strategies, risk preference, and the objective. Section \ref{SinglePeriodOpt} gives our first result for single period problem. Section \ref{MultiPeriod} provides the result for the multi-period setting. In Section \ref{numerics}, we present numerical examples. The paper concludes in Section \ref{conclusion}. The proofs of the results are presented in the Appendix. \\

\section{Model}
\label{Model} 
In this paper, we consider a multi-period stochastic model of investment. The randomness is associated with three discrete time random walks $\epsilon_n^{s}, \epsilon_n^{I}$ and $\epsilon_n^{b}$ defined on a complete probability space $(\Omega,\mathscr{F}t_{N},\left \{\mathscr{F}_{t_{n}}\right \},\mathbb{P})$. The random walks are symmetric in $\mathbb{P}$ and  
\begin{equation}
\mathbb{P}(\epsilon_n^{j}=\pm 1)=1/2, \quad n=0,1,\ldots,N,\,\,\, j={s,I,b}.
\end{equation}
There are two securities available for trading, the stock and the bond. We take bond as numeraire and thus assume it offers zero interest rate. There is a stream of income for the agent and a benchmark index representing the state of the economy. \\

The trading horizon is $[0,T]$, which we divide into $N$ trading dates $t=0,h,...,(N-1)h$. Here $h=T/N$, and $t_n=nh$ for $n=0,1,2,...,N$. The stock price process $S :=\{S_t: t=0,h,...,(N-1)h\}$ follows the difference equation 
\begin{equation}
\Delta S_{t_n}=S_{t_n}(\mu_1h + \sigma_1\sqrt{h} \epsilon_n^{s}),
\end{equation}
where $S_0>0$ and $\Delta S_{t_n}=S_{t_{n+1}}-S_{t_{n}}$. Drift parameter $\mu_1$ and volatility parameter $\sigma_1$ are chosen so that the stock price remains positive. The agent in this economy also receives random income at every time point. Income process $I :=\{I_t: t=0,h,...,(N-1)h\}$ of the agent is given by 
\begin{equation}
\Delta I_{t_n}=e^{k t_{n}}|\mu_2h + \sigma_2\sqrt{h} \epsilon_n^{I}|.
\end{equation}
The benchmark index $B :=\{B_t: t=0,h,...,(N-1)h\}$, represents the performance of the economy and is correlated to the underlying stock and the income. It is given by the following difference equation: 
\begin{equation}
\Delta B_{t_n}=B_{t_n}(\mu_3h + \sigma_3\sqrt{h} (a_{31}\epsilon_n^{s} + a_{32}\epsilon_n^{I} + a_{33}\epsilon_n^{b}))
\end{equation}
where $a_{31}>0$, $a_{32}>0$ and $a_{33}=\sqrt{1-a_{31}^2-a_{32}^2}$. As previously, drift $\mu_3$ and volatility $\sigma_3$ are chosen so that the benchmark index is positive. 

\subsection{Trading Strategies}
An investor in this model starts with initial wealth $X_0=x\in \mathbb{R}$ for investment in stocks. Let $\pi_n$ be the amount of wealth invested in stock at time $t_n:n=0,1,2,...,N-1$. The wealth from investments is governed by the self-financing equation 
\begin{equation}
\Delta X_{t_n}=\pi_n(\mu_1h + \sigma_1\sqrt{h} \epsilon_n^{s})\\
\end{equation}
\subsection{Risk Preference}The utility of the agent is assumed to be of exponential type with coefficient of absolute risk aversion given by $\gamma$. Specifically,
\begin{equation}
U(x)=-\exp(-\gamma x)%
\label{expUtility}
\end{equation}
\subsection{Objective Function}
The performance of the strategy $\pi$ is measured by the expected utility criterion applied to final wealth subject to a correlation constraint. Assume 
the current time is $t_n.$The final wealth $W_{t_N}$ is given by sum of wealth accumulated due to investment and random income. Thus, $W_{t_N}=X_{t_N}+I_{t_N}$.
\begin{equation}
\begin{aligned}
& \underset{\pi \in \Pi_{t_n}}{\text{maximize}}
& & \mathbb{E}^\mathbb{P}[-e^{-\gamma W_{t_N}^\pi}|\mathscr{F}_{t_{n}}] \\
& \text{subject to}
& & correl(W_{t_N},B_{t_N}|\mathscr{F}_{t_{n}}) \leq -\delta\\
\end{aligned}
\label{optimizationProblem}
\end{equation}
where $\Pi_{t_n}$ denotes the set of admissible trading strategies defined by 
\begin{equation}
\Pi_{t_n}:=\left\{\pi_{t_n},\pi_{t_{n+1}},...,\pi_{t_{N-1}}: \pi_{t_k}\in \mathscr{F}_{t_{k}} \right\},
\end{equation}
and the constraint in \eqref{optimizationProblem} is satisfied. In most of works the performance of a strategy is measured at $t_n=t_0=0.$
\subsection{Time inconsistency}
Non-linearity of the correlation constraint makes the optimization problem \eqref{optimizationProblem} time inconsistent. More precisely, the investor may have incentive to deviate from the optimal strategy they computed at past times. Thus, if at time $t_0$ optimal strategy is give by
\begin{equation}
\hat{\pi}=
\arg \left\{\begin{aligned}
& \underset{\pi \in \Pi_{t_0}}{\text{maximize}}
& & \mathbb{E}^\mathbb{P}[-e^{-\gamma W_{t_N}^\pi}] \\
& \text{subject to}
& & correl(W_{t_N},B_{t_N}) \leq -\delta\\
\end{aligned}\right\}
\end{equation}
then the optimal strategy evaluated at future time $t_n$ may not be the same. Thus, for times between [$t_n$,$t_N$] 
\begin{equation}
\hat{\pi}\neq
\arg \left\{\begin{aligned}
& \underset{\pi \in \Pi_{t_n}}{\text{maximize}}
& & \mathbb{E}^\mathbb{P}[-e^{-\gamma W_{t_N}^\pi}|\mathscr{F}_{t_{n}}] \\
& \text{subject to}
& & correl(W_{t_N},B_{t_N}|\mathscr{F}_{t_{n}}) \leq -\delta\\
\end{aligned}\right\}
\end{equation}

\subsection{Subgame Perfect Strategy}
This paper follows \cite{MTH2014} and \cite{PZ2013} in introducing the subgame perfect portfolio strategies as a substitute for optimal investment strategy. 
Let us proceed with a formal definition.

\begin{definition}

Let us assume that total time $t_N$=T is divided into h small time steps. So, total number of time steps is $N=T/h$. First, let us consider the time period [(N-1)h,Nh]. At time (N-1)h the optimization problem is:
\begin{equation}
\begin{aligned}
& \underset{\pi_{t_{N-1}}}{\text{maximize}}
& & E[-e^{-\gamma(X_{t_N}^\pi + I_{t_N})}|\mathscr{F}_{t_{N-1}}] \\
& \text{subject to}
& & correl(W_{t_N},B_{t_N}|\mathscr{F}_{t_{N-1}}) \leq -\delta
\end{aligned}
\tag{P1}
\label{P1}
\end{equation}
On time period [(N-2)h,Nh], the trading strategies $\pi$ are considered to be of the form
\[ \pi =
  \begin{cases}
    	\pi_{t_{N-1}}^*       & \quad \text{on }  [(N-1)h,Nh)\\
   	\pi_{t_{N-2}}  & \quad \text{on }  [(N-2)h,(N-1)h)\\
  \end{cases}
\]
for $\mathscr{F}_{t_{N-2}}$ measurable $\pi_{t_{N-2}}$ such that $(\pi_{t_{N-2}},\pi_{t_{N-2}}^*)\in \Pi_{t_{N-2}}$. The optimization problem is 
\begin{equation}
\begin{aligned}
& \underset{\pi_{t_{N-2}}}{\text{maximize}}
& & E[-e^{-\gamma(X_{t_N}^\pi + I_{t_N})}|\mathscr{F}_{t_{N-2}}] \\
& \text{subject to}
& & correl(W_{t_N},B_{t_N}|\mathscr{F}_{t_{N-2}}) \leq -\delta\\
\end{aligned}
\tag{P2}
\label{P2}
\end{equation}
Now, we can proceed iteratively. For the period [(N-n)h,Nh], the trading strategies $\pi$ considered are of the form
\[ \pi =
  \begin{cases}
    	\pi_{t_k}^*       & \quad \text{when } k= N-n+1,N-n+2,...,N-1\\
   	\pi_{t_k}  & \quad \text{when }  k= N-n\\
  \end{cases}
\]
for $\mathscr{F}_{t_{N-n}}$ measurable $\pi_{t_{N-n}}$ such that $(\pi_{t_{N-n}},\left\{\pi_{t_k}^*\right\}_{k=N-n+1,N-n+2,...,N-1})\in \Pi_{t_{N-n}}$. The optimization problem is
\begin{equation}
\begin{aligned}
& \underset{\pi_{t_{N-n}}}{\text{maximize}}
& & E[-e^{-\gamma(X_{t_N}^\pi + I_{t_N})}|\mathscr{F}_{t_{N-n}}] \\
& \text{subject to}
& & correl(W_{t_N},B_{t_N}|\mathscr{F}_{t_{N-n}}) \leq -\delta\\
\end{aligned}
\tag{Pn}
\label{Pn}
\end{equation}
Thus, the subgame perfect strategy is $\pi^*=(\pi_{t_0}^*,\pi_{t_1}^*,\pi_{t_2}^*,....,\pi_{t_{N-1}}^*).$
\end{definition}

With the portfolio optimization problems in place, we are ready to state our first result.

\section{Single Period Optimization}
\label{SinglePeriodOpt}

\textbf{Standing Assumption} Within this section we assume that $\delta \in (0,a_{31})$ because whenever $\delta>a_{31}$ there are no admissible trading strategies. The following lemma gives the correlation for the single period optimization problem. Let
\begin{equation}
M=\frac{e^{Tk}}{2}(|\mu_2h + \sigma_2\sqrt{h}|+ |\mu_2h - \sigma_2\sqrt{h}|)
\end{equation}
\begin{lemma}
The correlation between total wealth $W_{t_N}$ and the index $B_{t_N}$ conditional on $\mathscr{F}_{t_{N-1}}$ is given by
\begin{equation}
correl(W_{t_N},B_{t_N}|\mathscr{F}_{t_{N-1}})=\frac{\pi_{t_{N-1}} \sigma_1 a_{31}\sqrt{h} + Ma_{32}}{\sqrt{\pi_{t_{N-1}}^2 \sigma_1^2h + M^2}}
\end{equation}

\label{lemmaSingleP}
\end{lemma}
\begin{proof}
See Appendix A
\end{proof}

Now, we are in right spot to state the first theorem on optimal investment for single period model. First we state the result for the unconstrained case.
\begin{theorem}\label{un}
 The subgame perfect strategy $\bar{\pi}$ of the unconstrained portfolio optimization problem  is given by
\begin{equation}
\bar\pi =\arg \underset{\pi_{t_{N-1}}}{\text{max}} E[-e^{-\gamma(X_{t_N}^\pi + I_{t_N})}|\mathscr{F}_{t_{N-1}}] = \frac{1}{2\gamma\sigma_1\sqrt{h}}\ln(\frac{1+\theta\sqrt{h}}{1-\theta\sqrt{h}}).
\label{eq:pibar_unconstraint}
\end{equation}
\end{theorem}
\begin{proof}
See Appendix B
\end{proof}

Let 
\begin{equation}
R_1 :=\frac{M}{\sigma_1\sqrt{h}(a_{31}^2-\delta^2)}(-a_{31}a_{32}-\delta\sqrt{a_{31}^2+a_{32}^2-\delta^2}),\qquad
\theta=\frac{\mu_1}{\sigma_1}. 
\end{equation}
Let us move to the constrained case.
\label{theoremSinglePeriod}
\begin{theorem}
 The subgame perfect strategy $\pi_{t_{N-1}}^*$ of the portfolio optimization problem \ref{P1} is given by
\begin{equation}
 \pi_{t_{N-1}}^* = \min \left\{\bar\pi,R_1\right\}.
\end{equation}
\label{theoremSinglePeriod}
\end{theorem}
\begin{proof}
See Appendix C
\end{proof}

\begin{remark}
Due to the assumption of exponential utility, the optimal strategy $\pi_{t_{N-n}}^*$ is independent of the wealth at time $t_{N-n}$. Furthermore, if the constraint is binding, the optimal strategy is also independent of the risk aversion parameter $\gamma$. 
\end{remark}

\section{Multi Period Optimization}
\label{MultiPeriod}

In this subsection, we will look at the multiperiod optimization problem \ref{Pn}. Since the subgame perfect strategy is achieved using backward induction, imagine that we know the portfolio strategy for every time after $t_{N-n}$ and we want to find the optimal strategy at $t_{N-n}$. Let
\begin{align*}
\theta_3&=\frac{\sigma_3\sqrt{h}}{\tilde{\mu_3}} \label{eq22},\quad {\tilde{\mu_3}}=1+h\mu_3,\quad b_1=\theta_3\sigma_1a_{31}\sqrt{h},\,\,\\
k_{1,n}^2&=[(1+\theta_3^2)^n-1]\sigma_1^2h \quad \text{for} \  n=1,2,3,\ldots,N
\end{align*}
\textbf{Standing Assumption}  Within this section we will assume that $\delta \in (0,\frac{b_1}{k_{1,N}})$  
since whenever $\delta>\frac{b_1}{k_{1,N}}$ there are no admissible trading strategies. Before moving further, we will first compute the correlation between the final wealth and index given information till time $t_{N-n}.$

Let 
\begin{align}
&b_1:=\theta_3\sigma_1a_{31}\sqrt{h}\\
&k_{1,n}^2:=[(1+\theta_3^2)^n-1]\sigma_1^2h\\
&\M:=\frac{|\mu_2h + \sigma_2\sqrt{h}|-|\mu_2h - \sigma_2\sqrt{h}|}{2}\\
&b_{2,N-n} (\pi):=\theta_3[\M a_{32}\sum_{s=N-n+1}^{N}e^{kt_s} + \sigma_1a_{31}\sqrt{h}\sum_{i=N-n+1}^{N-1}\pi_{i}]\\
&k_{2,N-n}^2 (\pi):=[(1+\theta_3^2)^n-1]\{\M^2\sum_{s=N-n+1}^{N}e^{2kt_s} + \sigma_1^2h\sum_{i=N-n+1}^{N-1}\pi_{t_{i}}^2\}
\end{align}

\begin{lemma}
The correlation between total wealth $W_{t_N}$ and the index $B_{t_N}$ conditional on $\mathscr{F}_{t_{N-n}}$ is given by
\begin{equation}
correl(X_N+I_N,B_N|\mathscr{F}_{t_{N-n}})=\frac{b_1\pi_{t_{N-n}}+b_{2,N-n}}{\sqrt{k_{1,n}^2\pi_{t_{N-n}}^2+k_{2,N-n}^2}}.
\end{equation}
\label{lemmaMultiPeriod}
\end{lemma}
\begin{proof}
Details in Appendix D
\end{proof}

Next we present the central result of the paper, i.e., the subgame perfect strategy for multi-period model.
\begin{theorem}
The subgame perfect strategy of the portfolio optimization problems \ref{P1}, \ref{P2},...\ref{Pn} for n=1,2,...N is 
\begin{equation}
 \pi_{t_{N-n}}^* = \min \left\{\bar\pi,R_{1,N-n}\right\},
\end{equation}
where $R_{1,N-n}$ is the left root of the quadratic equation
\begin{equation}
Q(x)=x^2(b_1^2-k_{1,n}^2\delta^2) + 2b_1b_{2,N-n}(\pi^*)x + ([b_{2,N-n}(\pi^*)]^2-\delta^2[k_{2,N-n}(\pi^*)]^2)=0
\end{equation}
Moreover, $b_1$, $b_{2,N-n}$, $k_{2,N-n}$, $k_{1,n}$ and $\theta_3$ are defined above.
\label{mainTheorem}
\end{theorem}
\begin{proof}
Details in Appendix E
\end{proof}

\subsection{Pre-commitment Strategy}
The subgame pefect strategy of Theorem \ref{mainTheorem} led to constant $\pi_{t_{N-n}}^*$ for $\forall n$ $\in$ $\{1,2,...,N \}$. This leads us to search for a time $t = t_0 = 0$ deterministic optimal strategy, we call \textit{precommitment  strategy}, formally defined below 

\begin{definition}
The time zero deterministic optimal strategy is called the precommitment strategy and it is denoted by \textbf{$\hat{\Pi}$}. 
Thus, it is the solution of
\begin{equation}
\begin{aligned}
& \underset{  \pi_{t_{0}},\pi_{t_{1}},...\pi_{t_{N-1}}\,\mbox{deterministic} }{\text{maximize}}
& & E[-e^{-\gamma(X_{t_N}^\pi + I_{t_N})}] \\
& \text{subject to}
& & correl(W_{t_N},B_{t_N}) \leq -\delta.
\end{aligned}
\label{precommitmentObj}
\end{equation}
\end{definition}

As in Theorem 4.1 let
$$ k_1 := k_{1,N}, \quad b_2 := b_{2,0}(\pi),\quad k_2 := k_{2,0},\quad a:=1-\frac{\delta^2k_1^2}{b_1^2}.$$
For deterministic $\pi_{t_{0}},\pi_{t_{1}},...\pi_{t_{N-1}}$, we can reformulate the problem the above optimization problem as, 

\begin{equation}\label{preCT1}
\begin{aligned}
& \underset{\pi_{t_{0}},\pi_{t_{1}},...\pi_{t_{N-1}}}{\text{maximize}}
& & E[-e^{-\gamma(\sum_{i=0}^{N-1} \pi_i(\mu_1h + \sigma_1\sqrt{h}\epsilon_i^s) )}] \\
& \text{subject to}
& &  \frac{b_1\sum_{i=0}^{N-1}\pi_i + b_2}{\sqrt{k_1^2\sum_{i=0}^{N-1}\pi_i^2 + k_2^2}} \leq -\delta.
\end{aligned}
\end{equation}
Let
\begin{equation}
c= \frac{b_2}{b_1} \text{ , } h^2=\frac{k_2^2\delta^2}{b_1^2} \text{ and } b=\frac{-c}{a+N-1}\\
\end{equation}
\label{thmpct}
Finally we state our next theorem on the \textit{deterministic precommittement strategy}.
\begin{theorem}\label{M2}
The \textit{deterministic precommittement strategy} is given by 
\begin{equation}
\hat{\pi}_{t_i} = \min\lbrace  \bar{\pi},-b -\sqrt{\frac{c^2(1-a)}{N(a+N-1)^2} + \frac{h^2}{{N(a+N-1)}}} \rbrace \text{ },\,\,\, i \in \{0,1,...,N-1 \}.
\end{equation}
\end{theorem}

\begin{proof}
Appendix F
\end{proof}

Our last result concerns the optimal unconstrained strategies and is summarized by the following theorem.

\begin{theorem}\label{unm}
 In the unconstrained case, the subgame perfect strategy and the precommittement strategy are equal, and they are given by 
 
 \begin{equation}
\hat{\pi}_{t_i}=\pi^*_{t_i} = \bar\pi = \frac{1}{2\gamma\sigma_1\sqrt{h}}\ln(\frac{1+\theta\sqrt{h}}{1-\theta\sqrt{h}}), \text{ }\,\,\, i \in \{0,1,...,N-1 \}.
\label{eq:pibar_unconstraintm}
\end{equation}
\end{theorem}

\begin{proof}
Appendix G
\end{proof}

\section{Numerical Simulation}
\label{numerics}

In this section we use numerical simulations to compare the performance of three strategies, (i) CSGP  (ii) CPC, and (iii) UnSGP  under different scenarios. The base parameters for the experiment are given in Table~\ref{tab:Parameters}. Return of the strategy is measured via the expected terminal utility and the corresponding risk via $5^{th}$ percentile of the terminal wealth, defined for $i^{th}$ strategy as: $$P^{0.05}_i = \mathrm{perctile}\Big( W_{T}^i,  \ 5\% \Big), \quad i \in \{\mathrm{UnSGP, CSGP, CPC}\},$$ 
Intuitively, $P^{0.05}_i < X_0 + I_0 $ implies that the strategy lost money and $P^{0.05}_i > X_0 + I_0 $ reflects the profit at even the $5^{th}$ percentile.
\begin{table}[!ht]
\centering
\begin{tabular}{c} 
\hline
$\mu_1 =  7\%, \ \mu_2 =  3\%, \  \mu_3 =  5\%, \sigma_1 =  30\%, \ \sigma_2 =  10\%, \  \sigma_3 =  25\%$ \\ \hline
$a_{31} = a_{32} =  0.6, \ \gamma = 0.5, \ k = 0, \ \delta = 9\%$  \\ \hline
$N = 8, \  X_0 = I_0 = B_0 = 1, \ h = 1,  \ N_{sim} = 10 \text{ (millions)}$ \\ \hline
\end{tabular}
\caption{Parameters for the simulations.}
\label{tab:Parameters}
\end{table}

\textbf{Portfolio Investment:} Figure~\ref{fig:pi} depicts the portfolio investment of the CSGP and the CPC strategy. These graphs present several interesting features of the two strategies.  First, the negative correlation constraint results in short position in stock for CSGP and CPC strategy, this is in contrast to the long position for the UnSGP strategy. Secondly, the investments in CSGP strategy depends only on the number of time steps remaining. As a result, the investment remains the same in the last time step regardless of whether $N =$ 2, 4, 6, 8 or 10. On the contrary, the optimal investment in CPC strategy depends on the length of the investment horizon. Thirdly, for CSGP strategy the absolute value of the short position at time $ t = 0$ is lower for short investment horizon, i.e. $|\pi_{t_0}|$ for $N = 2$ is smaller than $|\pi_{t_0}|$ for $N = 10$. This is because the shorter time horizon allows the investor to better control the correlation constraint at the terminal step. Finally, the investments in the CPC strategy is approximately the time average of the investments for the CSGP strategy. This relationship results in the corresponding expected terminal utility of the two strategies to be almost the same. On the contrary, \cite{Wu2014} report higher utility for the CPC strategy. In our problem, the choice of exponential utility function makes the optimal investment strategies to be  independent of the initial wealth which in-turn leads to similar performance of CPC and CSGP strategy.  Consequently, for the remaining simulations we present results for only CSGP and UnSGP strategy.



\begin{figure}[!ht]
  \centering
  \begin{subfigure}[b]{0.45\textwidth}
    \includegraphics[width=\textwidth]{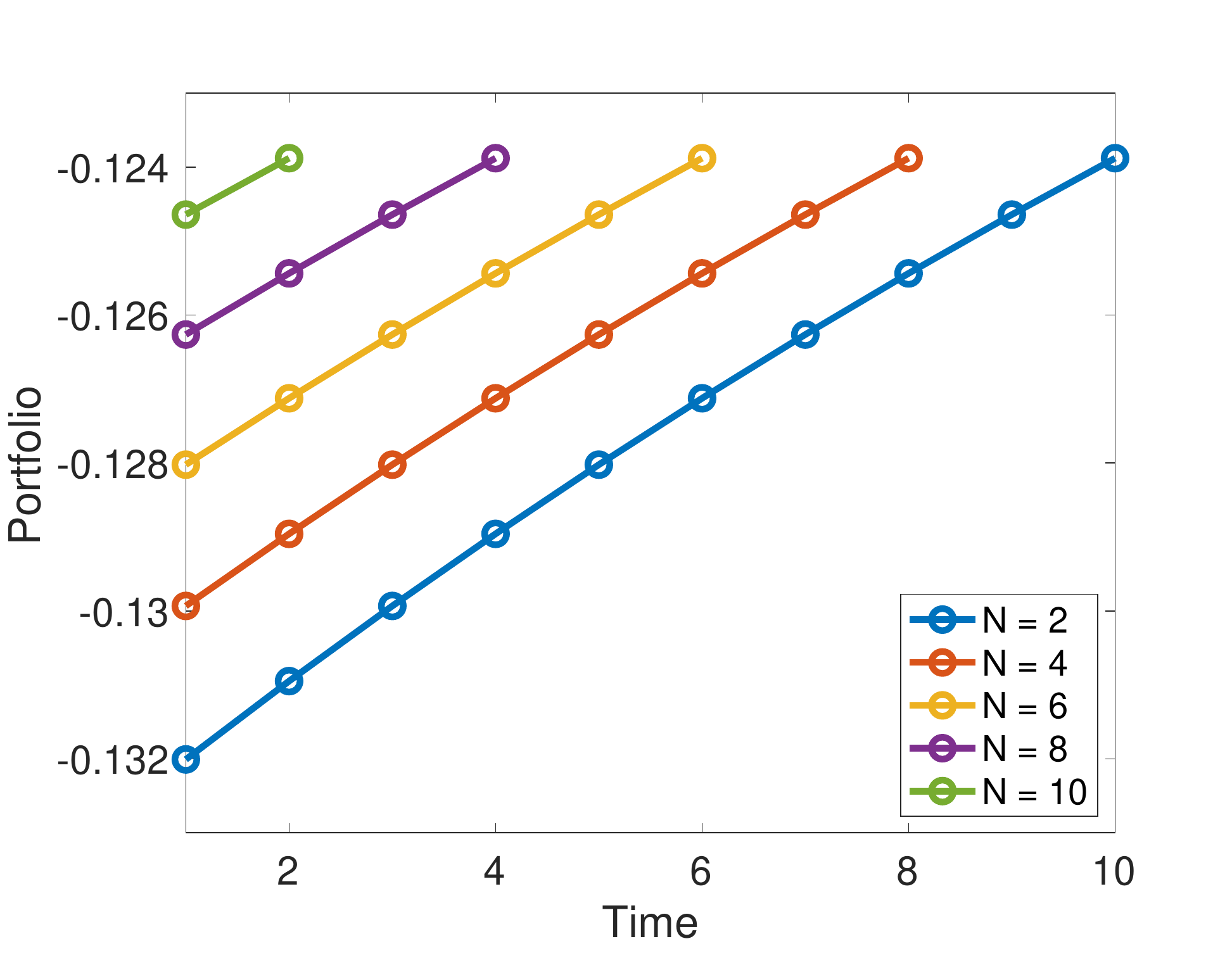}
    \caption{Constrained subgame perfect}
    \label{fig:pi_csgp}
  \end{subfigure}
 \quad
    \begin{subfigure}[b]{0.45\textwidth}
    \includegraphics[width=\textwidth]{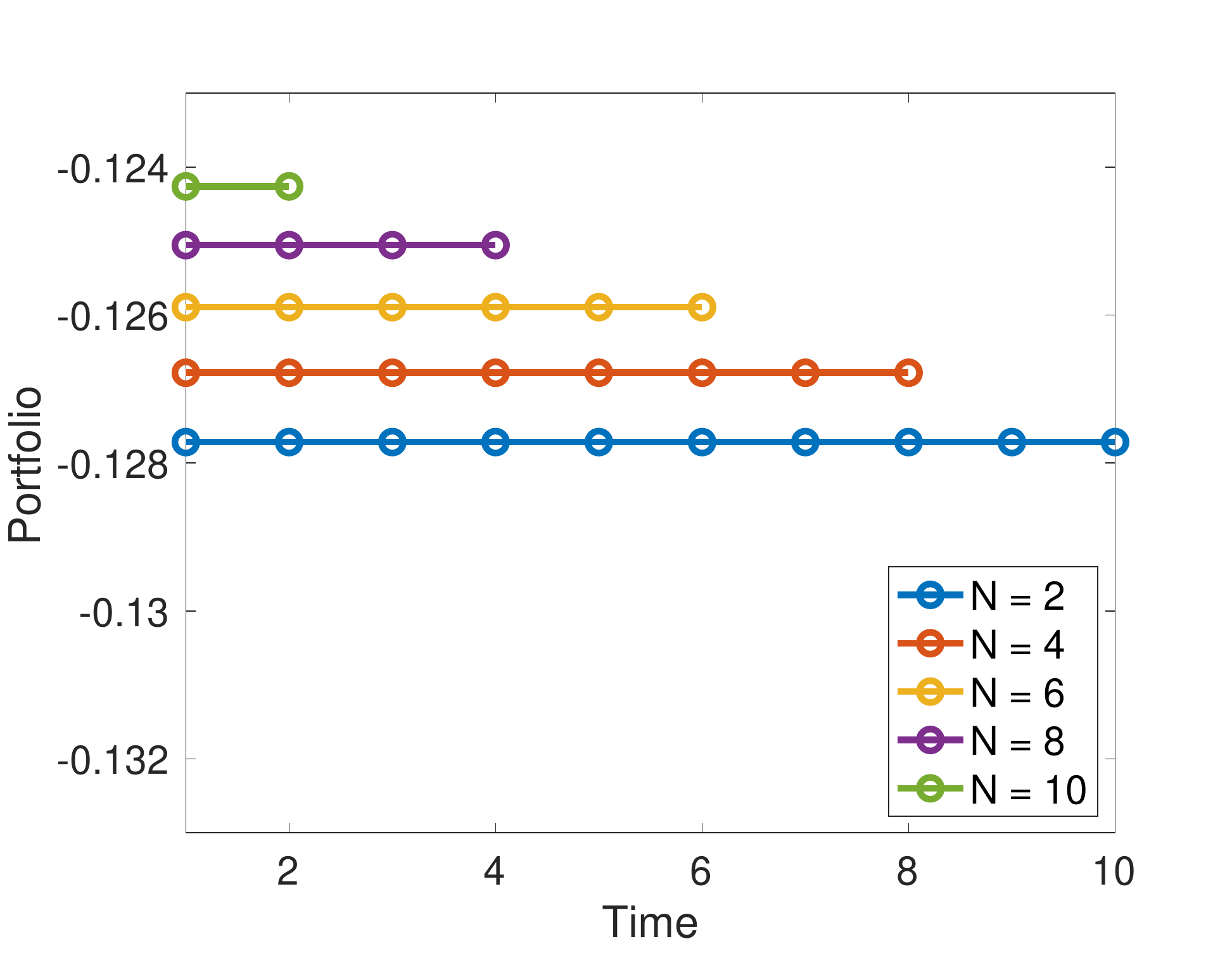}
    \caption{Constrained precommitment}
    \label{fig:pi_cpc}
  \end{subfigure}
\caption{Optimal portfolio investment}
\label{fig:pi}
\end{figure}

\textbf{Impact of correlation constraint: } Figure~\ref{fig:impact_delta} depicts the impact of correlation constraint at the terminal expected utility (and risk) for the CSGP strategy relative to the terminal expected utility (and risk) at $\delta = 1\%$. We find that a more conservative constraint (higher $\delta$) leads to lower utility and higher risk for different values of $N$. While lower utility for higher $\delta$ is expected, higher risk is driven by increased short position in stocks. For $N = 2$, we find the utility declines by 2\% and risk increases by 8\% as we increase from $\delta  = 1\%$ to $\delta = 40\%$. Longer investment horizon leads to higher reduction in utility and increase in risk due to increased uncertainty in stock price.\\
    \begin{figure*}[!ht]
        \centering
            \includegraphics[width=0.4\textwidth]{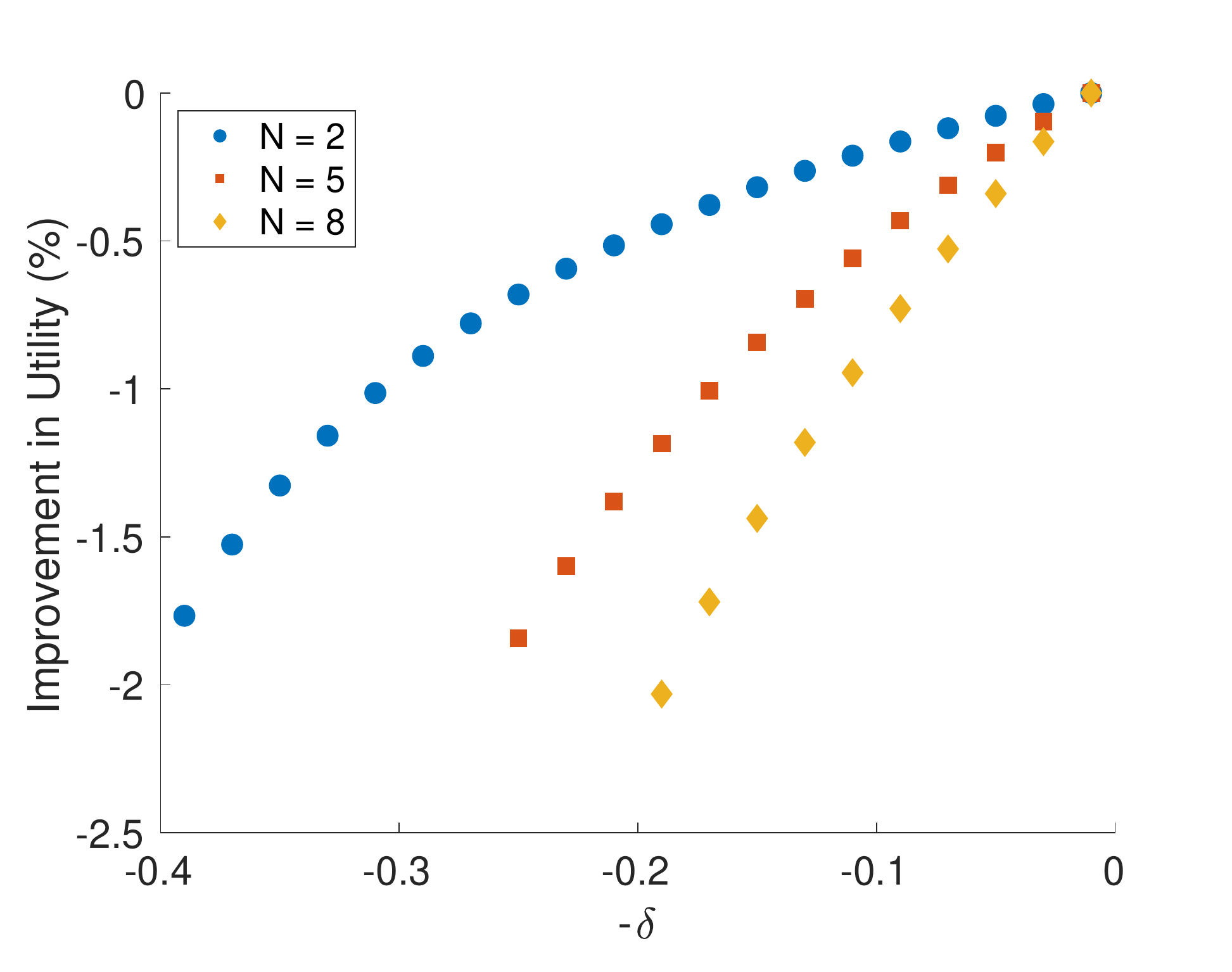}
            \includegraphics[width=0.4\textwidth]{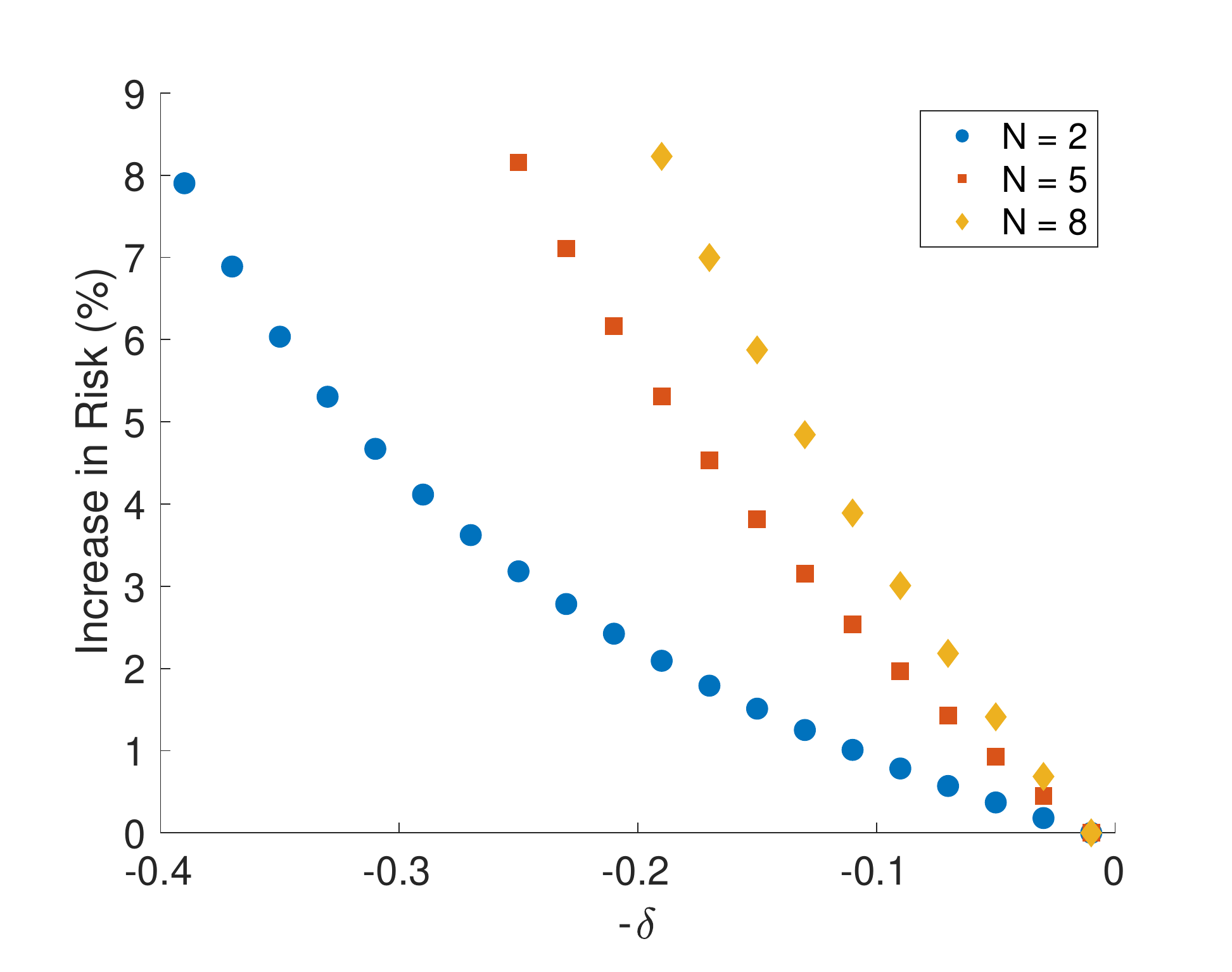}
        \caption[ ]
        {{\small Impact of $\delta$ on the CSGP strategy.}}
        \label{fig:impact_delta}
    \end{figure*}

\textbf{Performance in risky environment}
Next, we discuss the advantage of CSGP relative to UnSGP strategy as the investment scenario becomes more risky. We consider two risk scenarios: (i) Longer investment horizon (cf. Figure~\ref{fig:impact_delta}) and (ii) Higher correlation of terminal wealth for unconstrained portfolio with the market index. 
\begin{itemize}
\item \textbf{Longer investment horizon:} Figure~\ref{fig:impact_N} depicts the impact of investment horizon on the expected terminal utility and risk due to CSGP and UnSGP strategy. The UnSGP strategy leads to 50\% increase in utility and 150\% increase in risk of the portfolio as we increase investment horizon $N$ from  2 to 10. On the contrary, the CSGP strategy leads to only 30\% increase in utility but the risk \text{reduces} by 18\%.  Thus for the long investment horizon CSGP strategy not only improves utility but also provides considerable risk reduction. 
    \begin{figure*}[!ht]
        \centering
            \includegraphics[width=0.4\textwidth]{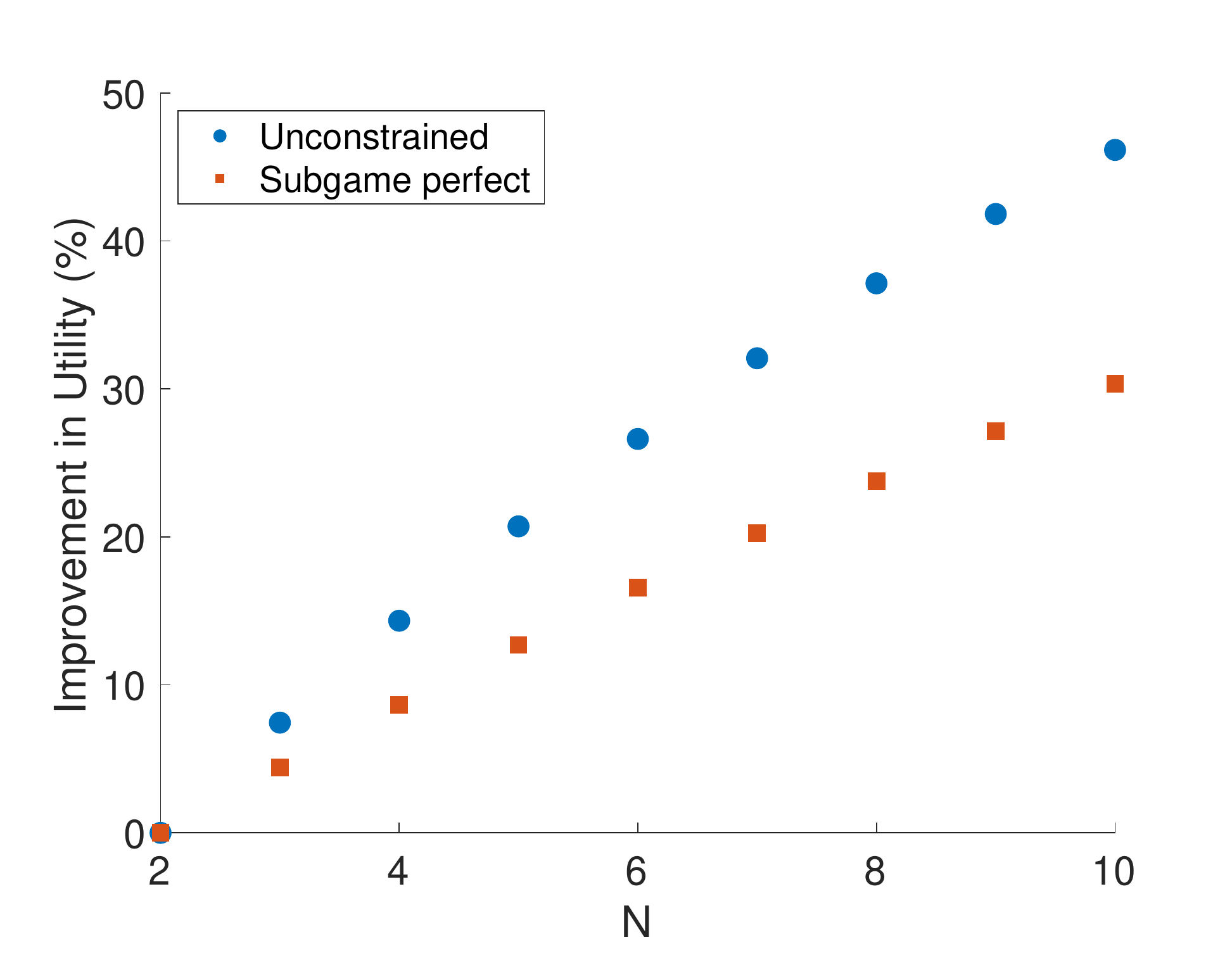}
            \includegraphics[width=0.4\textwidth]{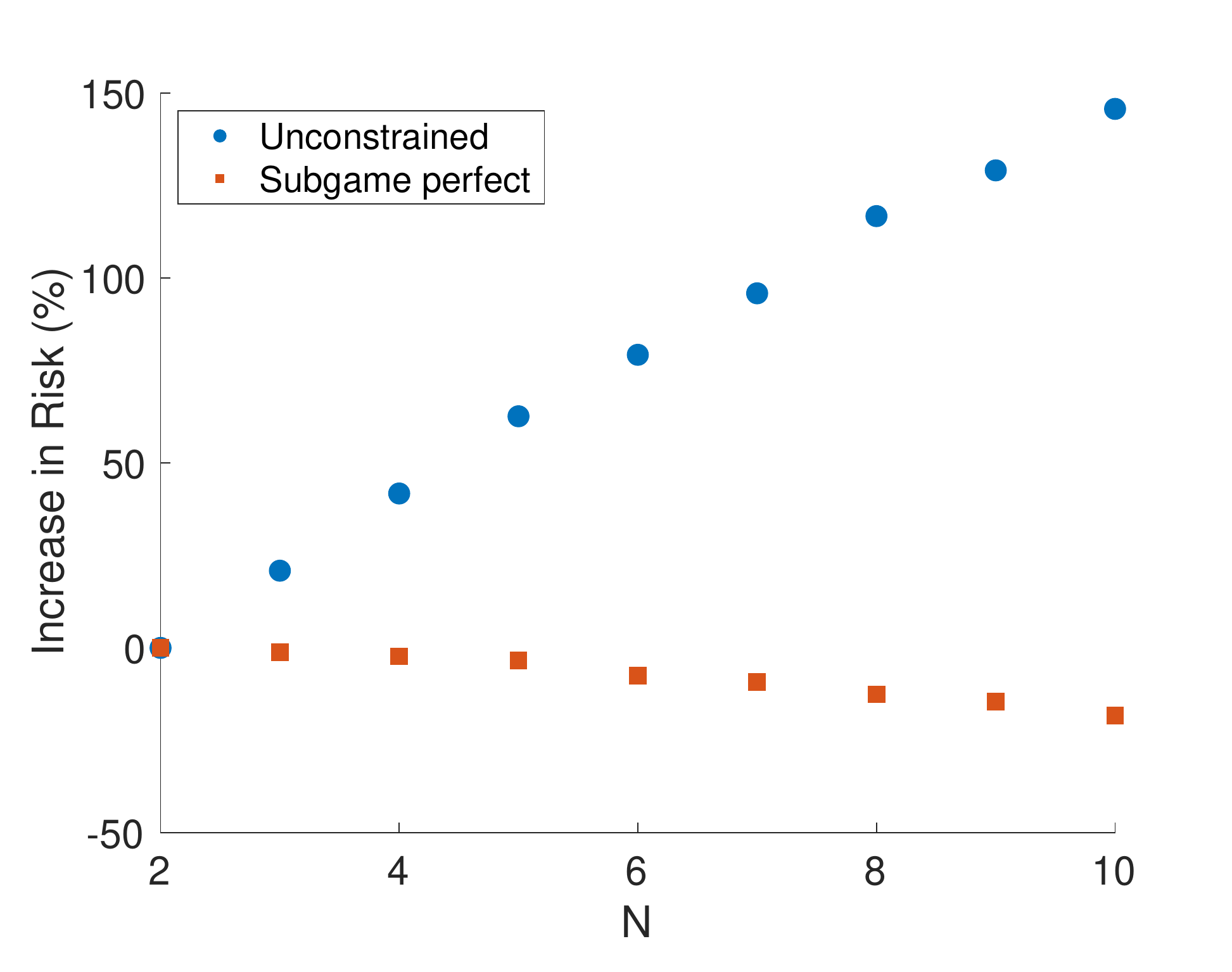}
        \caption[ ]
        {{\small Impact of investment horizon ($N$)}}
        \label{fig:impact_N}
    \end{figure*}
\item \textbf{Higher correlation of  unconstrained portfolio wealth with market index:} Figure~\ref{fig:impact_corrl} considers the impact of $\rho_{W_T^{\overline{\pi}}, B_T}$ on utility and risk. We consider different scenarios of $\rho_{W_T^{\overline{\pi}}, B_T}$ by changing the volatility of the stock $\sigma_1$ and the index $\sigma_3$. For simplicity we consider scenarios such that $\sigma := \sigma_1 = \sigma_3$. As  $\rho_{W_T^{\overline{\pi}}, B_T}$ increases, the UnSGP strategy provides 55\% improvement in utility, however,  it results in 250\% increase in risk relative to low correlation environment. For such high risk environment, the CSGP strategy provides considerable risk reduction. Unlike UnSGP strategy, here we observe 6\% reduction in risk at the expense of 5\% loss of utility.
\end{itemize}

    \begin{figure*}[!ht]
        \centering
            \includegraphics[width=0.4\textwidth]{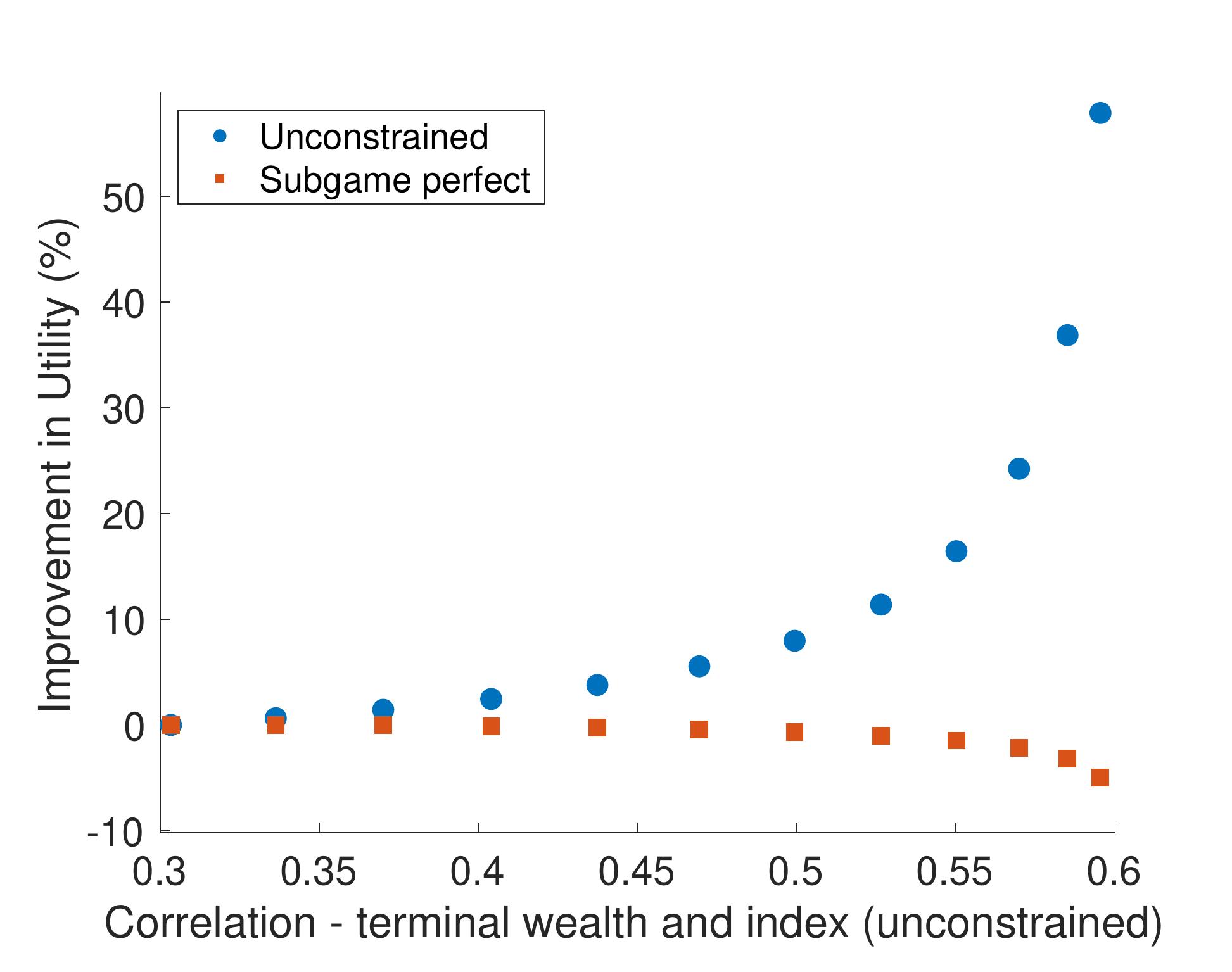}
            \includegraphics[width=0.4\textwidth]{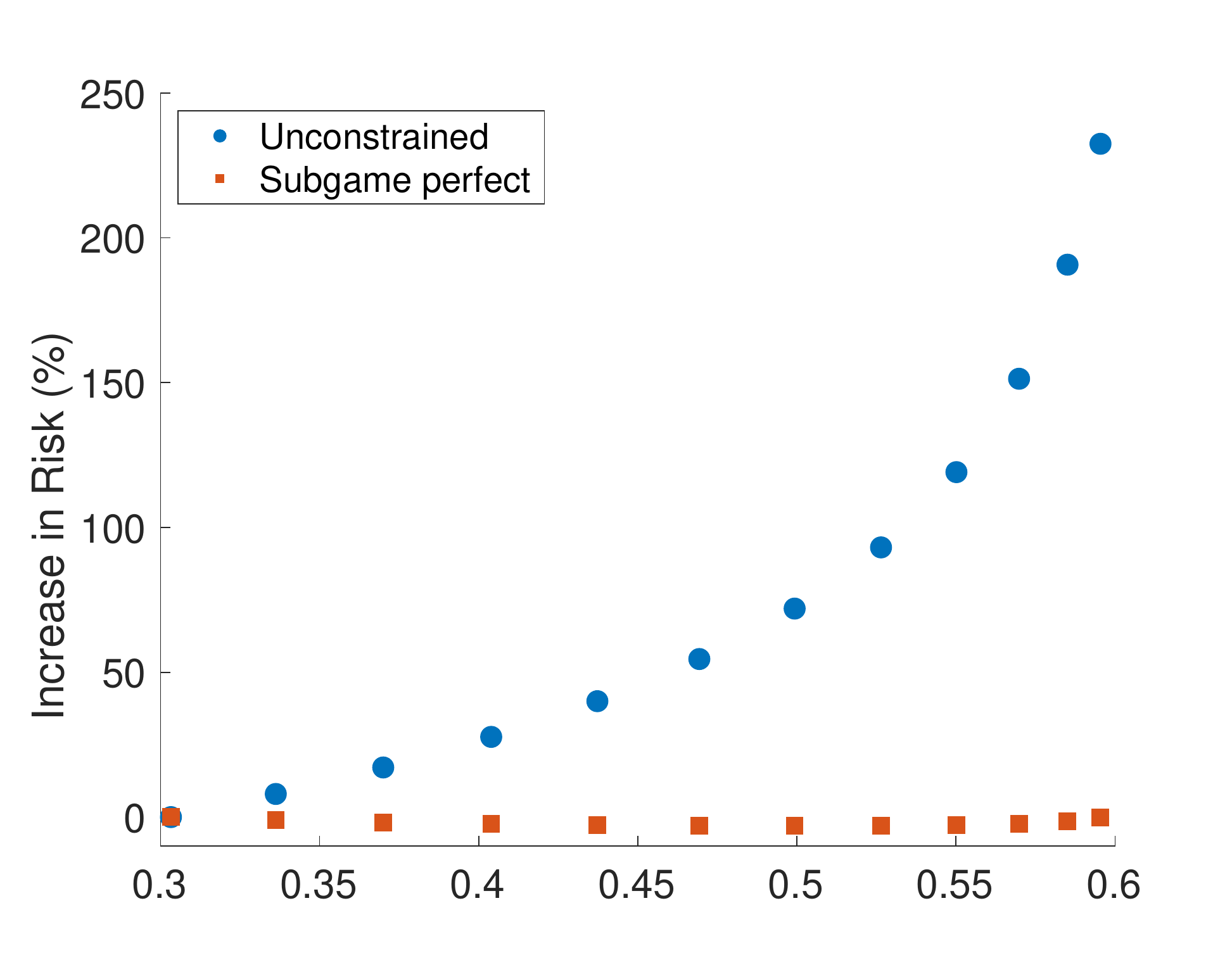}
        \caption[ ]
        {{\small High correlation environment}}
        \label{fig:impact_corrl}
    \end{figure*}

\section{Concluding remarks}
\label{conclusion}

We study the impact of correlation constraint on a portfolio optimization problem of one stock and stochastic income stream in a discrete multiperiod setting. The constraint  represents the correlation of the wealth of the portfolio with a benchmark index representing the state of the economy. The non-linearity due to correlation constraint results in time inconsistency of the optimal portfolios. A solution to this time consistency issue is the subgame perfect (CSGP) strategy introduced through backward induction. This strategy considers the portfolio problem as a non-cooperative game between players at each time point. We compare the performance of CSGP strategy to the constrained precommitment (CPC) and the unconstrained subgame perfect (UnSGP) strategy. Using numerical simulations we find that the correlation constraint provides significant risk reduction relative to the unconstrained strategy. The reduction in risk compensates for the loss in utility due to the constraint. We also find similar performances of the two constrained strategies CSGP and CPC.  Future works aim at extending the model to allow for more than one stock as well as convergence results to a continuous time model.


\appendix

\section{Proof of Lemma \ref{lemmaSingleP}}

For defining correlation we need some fundamental quantities 
like covariance between $W_{t_N}$, $B_{t_N}$ and variance of $W_{t_N}$ and $B_{t_N}$. Recall that $t_N=T.$ 
\begin{align}
Cov(W_{t_N},B_{t_N}|\mathscr{F}_{t_{N-1}})&=Cov(I_{t_N},B_{t_N}|\mathscr{F}_{t_{N-1}})+Cov(X_{t_N},B_{t_N}|\mathscr{F}_{t_{N-1}})\\
		    &=\pi_{t_{N-1}}B_{t_{N-1}}\sigma_1\sigma_3a_{31}h + e^{Tk}B_{t_{N-1}}\sigma_3a_{32}\sqrt{h}\{\frac{|\mu_2h + \sigma_2\sqrt{h}|- |\mu_2h - \sigma_2\sqrt{h}|}{2}\}\\
Var(W_{t_N}|\mathscr{F}_{t_{N-1}})&=V(X_{t_N}|\mathscr{F}_{t_{N-1}})+V(I_{t_N}|\mathscr{F}_{t_{N-1}})\\
	     &=\pi^2_{t_{N-1}}\sigma_1^2h + \frac{e^{2Nk}}{4}(|\mu_2h + \sigma_2\sqrt{h}|+ |\mu_2h - \sigma_2\sqrt{h}|)^2\\
Var(B_{t_N}|\mathscr{F}_{t_{N-1}})&=B_{t_{N-1}}^2\sigma_3^2h
\end{align}
Thus, correlation between $W_{t_N}$ and $B_{t_N}$ conditional on $\mathscr{F}_{t_{N-1}}$ is given by
\begin{equation}
correl(W_{t_N},B_{t_N}|\mathscr{F}_{t_{N-1}})=\frac{\pi_{t_{N-1}} \sigma_1 a_{31}\sqrt{h} + Ma_{32}}{\sqrt{\pi_{t_{N-1}}^2 \sigma_1^2h + M^2}}
\end{equation}
where, as before, $M=\frac{e^{Tk}}{2}(|\mu_2h + \sigma_2\sqrt{h}|+ |\mu_2h - \sigma_2\sqrt{h}|).$\\

\section{Proof of Theorem \ref{un} }
In this section, we will derive the optimal investment for the unconstrained problem. 
\begin{align}
\bar{\pi}=\bar{\pi}_{t_{N-1}}&=\arg \max E[-e^{-\gamma(X_{t_N} + I_{t_N})}|\mathscr{F}_{t_{N-n}}]\\
		&=\arg \min e^{-\gamma(X_{t_{N-1}} + I_{t_{N-1}})}E[e^{-\gamma(X_{t_N}-X_{t_{N-1}} + I_{t_N}-I_{t_{N-1}})}|\mathscr{F}_{t_{N-n}}]\\
		&=\arg \min e^{-\gamma(X_{t_{N-1}} + I_{t_{N-1}})}E[e^{-\gamma \pi_{t_{N-1}}(\mu_1h + \sigma_1\sqrt{h} \epsilon_N^{s})}|\mathscr{F}_{t_{N-n}}]E[e^{-\gamma e^{kt_N}|\mu_2h + \sigma_2\sqrt{h} \epsilon_N^{I}|}|\mathscr{F}_{t_{N-n}}]\\
		&=\arg\min e^{-\gamma(X_{t_{N-1}} + I_{t_{N-1}})} \{\frac{e^{-\gamma \pi_{t_{N-1}}(\mu_1h + \sigma_1\sqrt{h})}+e^{-\gamma \pi_{t_{N-1}}(\mu_1h - \sigma_1\sqrt{h})}}{2}\}C\\
		&=\arg\min \{e^{-\gamma \pi_{t_{N-1}}(\mu_1h + \sigma_1\sqrt{h})}+e^{-\gamma \pi_{t_{N-1}}(\mu_1h - \sigma_1\sqrt{h})} \}\\
		&=\frac{1}{2\gamma\sigma_1\sqrt{h}}\ln(\frac{1+\theta\sqrt{h}}{1-\theta\sqrt{h}}), 
\end{align}
where 

$$
\theta=\frac{\mu_1}{\sigma_1},\quad
C=\{\frac{e^{-\gamma e^{kt_N} |\mu_2h + \sigma_2\sqrt{h}|}+e^{-\gamma e^{kt_N} |\mu_2h - \sigma_2\sqrt{h}|}}{2}\}.
$$

\section{Proof of Theorem \ref{theoremSinglePeriod}}

\begin{proof}By Lemma \ref{lemmaSingleP} it follows that
$$correl(W_{t_N},B_{t_N}|\mathscr{F}_{t_{N-1}})=\frac{\pi_{t_N-1} \sigma_1\sqrt{h} a_{31} + Ma_{32}}{\sqrt{\pi_{t_N-1}^2 \sigma_1^2h + M^2}}.$$
Thus
\begin{align}
correl(W_{t_N},B_{t_N}|\mathscr{F}_{t_{N-1}})&=\frac{\pi_{t_N-1} \sigma_1\sqrt{h} a_{31} + Ma_{32}}{\sqrt{\pi_{t_N-1}^2 \sigma_1^2h + M^2}} \leq -\delta \\
		& \Leftrightarrow  (\frac{\pi_{t_N-1} \sigma_1\sqrt{h} a_{31} + Ma_{32}}{\sqrt{\pi_{t_N-1}^2 \sigma_1^2h + M^2}})^2 \geq\delta^2  \text{ and } \pi_{t_N-1} \sigma_1\sqrt{h} a_{31} + Ma_{32} \leq 0 \\
		& \Leftrightarrow  \pi_{t_N-1}^2 \sigma_1^2h(a_{31}^2-\delta^2)+2\pi_{t_N-1}\sigma_1\sqrt{h}a_{31}a_{32}M+M^2(a_{32}^2-\delta^2) \geq \label{qe}\\
		&\text{ and } \pi_{t_N-1} \leq \frac{-Ma_{32}}{\sigma_1\sqrt{h} a_{31}} 
\label{constraintSimpl}
\end{align}
By \eqref{qe} we get that
\begin{equation}
\pi_{t_N-1} \leq R_1 \text{ and } \pi_{t_N-1} \geq R_2
\end{equation}

where 
\begin{align}
R_1&=\frac{M}{\sigma_1\sqrt{h}(a_{31}^2-\delta^2)}(-a_{31}a_{32}-\delta\sqrt{a_{31}^2+a_{32}^2-\delta^2})\\
R_2&=\frac{M}{\sigma_1\sqrt{h}(a_{31}^2-\delta^2)}(-a_{31}a_{32}+\delta\sqrt{a_{31}^2+a_{32}^2-\delta^2})
\end{align}

Also, \eqref{constraintSimpl} tells us that $$\pi_{t_N-1} \leq \frac{-Ma_{32}}{\sigma_1\sqrt{h} a_{31}}.$$ To complete the proof it remains to show that 
\begin{equation}
R_1 < \frac{-Ma_{32}}{\sigma_1\sqrt{h} a_{31}} < R_2.
\end{equation} \\
Since, 
\begin{align}
&\frac{-Ma_{32}}{\sigma_1\sqrt{h}(a_{31}-\frac{\delta^2}{a_{31}})} < \frac{-Ma_{32}}{\sigma_1\sqrt{h} a_{31}}\\ \implies& \frac{-Ma_{32}}{\sigma_1\sqrt{h}(a_{31}-\frac{\delta^2}{a_{31}})} - \frac{M\delta\sqrt{a_{31}^2+a_{32}^2-\delta^2}}{\sigma_1\sqrt{h}(a_{31}^2-\delta^2)} <  \frac{-Ma_{32}}{\sigma_1\sqrt{h} a_{31}}\\
\implies& R_1 < \frac{-Ma_{32}}{\sigma_1\sqrt{h} a_{31}}.
\end{align}
For the claim $\frac{-Ma_{32}}{\sigma_1\sqrt{h} a_{31}} < R_2$, we take a slightly different route. Notice that
\begin{align}
R_2&=\frac{M}{\sigma_1\sqrt{h}(a_{31}^2-\delta^2)}(-a_{31}a_{32}+\delta\sqrt{a_{31}^2+a_{32}^2-\delta^2})\\
     &=\frac{-Ma_{31}a_{32}}{\sigma_1\sqrt{h}(a_{31}^2-\delta^2)} + \frac{M}{\sigma_1\sqrt{h}(a_{31}^2-\delta^2)}\delta\sqrt{a_{31}^2+a_{32}^2-\delta^2})	\\
     &=\frac{-Ma_{32}}{\sigma_1\sqrt{h}a_{31}(1-\frac{\delta^2}{a_{31}^2})}	+ \frac{Ma_{32}}{\sigma_1\sqrt{h}a_{31}(1-\frac{\delta^2}{a_{31}^2})}\delta\sqrt{\frac{1}{a_{31}^2}+\frac{1}{a_{32}^2}-\frac{\delta^2}{a_{31}^2a_{32}^2}}\\
    &=\frac{-Ma_{32}}{\sigma_1\sqrt{h} a_{31}}(\frac{1}{1-\frac{\delta^2}{a_{31}^2}})[1-\delta\sqrt{\frac{1}{a_{31}^2}+\frac{1}{a_{32}^2}-\frac{\delta^2}{a_{31}^2a_{32}^2}}]
\label{part1}
\end{align}

Since $$\frac{\delta^2}{a_{31}^2}<1,$$ 
\begin{align}
&\implies \frac{\delta^2}{a_{31}^2a_{32}^2}<\frac{1}{a_{32}^2}\\
&\implies \frac{1}{a_{32}^2}-\frac{\delta^2}{a_{31}^2a_{32}^2}>0\\
&\implies \frac{1}{a_{31}^2}+\frac{1}{a_{32}^2}-\frac{\delta^2}{a_{31}^2a_{32}^2}>\frac{1}{a_{31}^2}\\
&\implies \delta\sqrt{\frac{1}{a_{31}^2}+\frac{1}{a_{32}^2}-\frac{\delta^2}{a_{31}^2a_{32}^2}}>\delta\frac{1}{a_{31}}\\
&\implies 1-\delta\sqrt{\frac{1}{a_{31}^2}+\frac{1}{a_{32}^2}-\frac{\delta^2}{a_{31}^2a_{32}^2}}<1-\delta\frac{1}{a_{31}}\\
&\implies \frac{1}{1-\frac{\delta^2}{a_{31}^2}}[1-\delta\sqrt{\frac{1}{a_{31}^2}+\frac{1}{a_{32}^2}-\frac{\delta^2}{a_{31}^2a_{32}^2}}]<\frac{1}{1-\frac{\delta^2}{a_{31}^2}}[1-\delta\frac{1}{a_{31}}]\\
&\implies \frac{1}{1-\frac{\delta^2}{a_{31}^2}}[1-\delta\sqrt{\frac{1}{a_{31}^2}+\frac{1}{a_{32}^2}-\frac{\delta^2}{a_{31}^2a_{32}^2}}]<\frac{1}{1+\frac{\delta}{a_{31}}}\\
&\implies \frac{-Ma_{32}}{\sigma_1\sqrt{h} a_{31}}\frac{1}{1-\frac{\delta^2}{a_{31}^2}}[1-\delta\sqrt{\frac{1}{a_{31}^2}+\frac{1}{a_{32}^2}-\frac{\delta^2}{a_{31}^2a_{32}^2}}]>-\frac{Ma_{32}}{\sigma_1\sqrt{h} a_{31}}\frac{1}{1+\frac{\delta}{a_{31}}}\\
&\implies -\frac{Ma_{32}}{\sigma_1\sqrt{h} a_{31}}\frac{1}{1+\frac{\delta}{a_{31}}} < R_2\\
&\implies -\frac{Ma_{32}}{\sigma_1\sqrt{h} a_{31}}<-\frac{Ma_{32}}{\sigma_1\sqrt{h} a_{31}}\frac{1}{1+\frac{\delta}{a_{31}}} < R_2
\end{align}

Thus, we have proved that $$R_1 < \frac{-Ma_{32}}{\sigma_1\sqrt{h} a_{31}} < R_2.$$ Hence $R_1$ will be optimal solution whenever $\bar\pi_{t_N-1} \geq R_1$, else $\bar\pi_{t_N-1}$ is the optimal solution. It is easy to observe that $$\frac{\pi_{t_N-1} \sigma_1\sqrt{h} a_{31} + Ma_{32}}{\sqrt{\pi_{t_N-1}^2 \sigma^2 + M^2}}>-a_{31}.$$ So, there is no admissible trading strategies if $\delta>a_{31}.$
\end{proof}

\section{Proof of lemma \ref{lemmaMultiPeriod}}
 We need some basic results on variance and covariance of terms involved. Straightforward computations lead to, 
\begin{align}
Var(X_{t_{N}}|\mathscr{F}_{t_{N-n}})&=\sigma_1^2h\sum_{i=N-n}^{N-1}\pi_{t_{i}}^2\\
Var(B_{t_{N}}|\mathscr{F}_{t_{N-n}})&=B_{t_{N-n}}^2[(\tilde{\mu_3}^2+\sigma_3^2h)^n-\tilde{\mu_3}^{2n}]\\
Var(I_{t_{N}}|\mathscr{F}_{t_{N-n}})&=\M^2\sum_{s=N-n+1}^{N}e^{2kt_s}\\
Cov(X_N,B_N|\mathscr{F}_{t_{N-n}})&=B_{t_{N-n}}\tilde{\mu_3}^{n-1}h\sigma_1\sigma_3a_{31}\sum_{i=N-n}^{N-1}\pi_{i}\\
Cov(I_N,B_N|\mathscr{F}_{t_{N-n}})&=B_{t_{N-n}}\M\tilde{\mu_3}^{n-1}\sqrt{h}\sigma_3a_{32}\sum_{s=N-n+1}^{N}e^{kt_s}
\end{align}
where $\M=\frac{|\mu_2h + \sigma_2\sqrt{h}|-|\mu_2h - \sigma_2\sqrt{h}|}{2}$. Thus,
\begin{align}
correl(X_N+I_N,B_N|\mathscr{F}_{t_{N-n}})&=\frac{Cov(X_N,B_N|\mathscr{F}_{t_{N-n}})+Cov(I_N,B_N|\mathscr{F}_{t_{N-n}})}{\sqrt{Var(X_{t_{N}}+I_{t_{N}}|\mathscr{F}_{t_{N-n}}) Var(B_{t_{N}}|\mathscr{F}_{t_{N-n}})}}\\
&=\frac{b_1\pi_{t_{N-n}}+b_{2,N-n}}{\sqrt{k_{1,n}^2\pi_{t_{N-n}}^2+k_{2,N-n}^2}}\\
&b_1=\theta_3\sigma_1a_{31}\sqrt{h}\\
&k_{1,n}^2=[(1+\theta_3^2)^n-1]\sigma_1^2h\\
&b_{2,N-n}(\pi)=\theta_3[\M a_{32}\sum_{s=N-n+1}^{N}e^{kt_s} + \sigma_1a_{31}\sqrt{h}\sum_{i=N-n+1}^{N-1}\pi_{i}]\\
&k_{2,N-n}^2(\pi)=[(1+\theta_3^2)^n-1]\{\M^2\sum_{s=N-n+1}^{N}e^{2kt_s} + \sigma_1^2h\sum_{i=N-n+1}^{N-1}\pi_{t_{i}}^2\}\\
&\theta_3=\frac{\sigma_3\sqrt{h}}{\tilde{\mu_3}}.
\end{align}

\section{Proof of theorem \ref{mainTheorem}}
\begin{proof}
By Lemma~\ref{lemmaMultiPeriod} it follows that 
\begin{equation*}
correl(X_N+I_N,B_N|\mathscr{F}_{t_{N-n}})=\frac{b_1\pi_{t_{N-n}}+b_{2,N-n}}{\sqrt{k_{1,n}^2\pi_{t_{N-n}}^2+k_{2,N-n}^2}}.
\end{equation*}
Thus, the constraint
\begin{align}
correl(X_N+I_N,B_N|\mathscr{F}_{t_{N-n}}) & =\frac{b_1\pi_{t_{N-n}}+b_{2,N-n}}{\sqrt{k_{1,n}^2\pi_{t_{N-n}}^2+k_{2,N-n}^2}}  \leq -\delta. \\
& \Leftrightarrow Q(\pi_{t_{N-n}}) \geq 0 \text{ and } \pi_{t_{N-n}} \leq \frac{-b_{2,N-n}}{b_1},
\end{align}
where $Q(\pi_{t_{N-n}}):=\pi_{t_{N-n}}^2(b_1^2-k_{1,n}^2\delta^2) + 2b_1b_{2,N-n}\pi_{t_{N-n}} + (b_{2,N-n}^2-\delta^2k_{2,N-n}^2).$ Similar to proof of  Theorem~\ref{theoremSinglePeriod}, we conclude this proof by proving that $\frac{-b_{2,N-n}}{b_1}$ lies between the roots of quadratic equation $Q(\pi_{t_{N-n}}) = 0$ for all $b_{2,N-n}$, i.e. 
$$\bar{R}_{1,N-n}\leq\frac{-b_{2,N-n}}{b_1}\leq \bar{R}_{2,N-n}, $$
where $\bar{R}_{1,N-n}$ and $\bar{R}_{2,N-n}$ are the roots of the equation $Q(\pi_{t_{N-n}}) = 0$, and are given by 
\begin{align}
\bar{R}_{1,N-n}&=\frac{-b_1b_{2,N-n}-\delta k_1k_{2,N-n}\sqrt{\big(\frac{b_{2,N-n}}{k_{2,N-n}}\big)^2+\big(\frac{b_1}{k_{1,n}}\big)^2-\delta^2}}{b_1^2-k_{1,n}^2\delta^2}\\
\bar{R}_{2,N-n}&=\frac{-b_1b_{2,N-n}+\delta k_1k_{2,N-n}\sqrt{\big(\frac{b_{2,N-n}}{k_{2,N-n}}\big)^2+\big(\frac{b_1}{k_{1,n}}\big)^2-\delta^2}}{b_1^2-k_{1,n}^2\delta^2}.
\end{align}  
We consider two cases (i) $b_{2,N-n}>0$ and (ii) $b_{2,N-n}<0:$ separately. \\
\textbf{CASE I:} $b_{2,N-n}>0:$
\begin{align}
\bar{R}_{1,N-n}&=\frac{-b_1b_{2,N-n}-\delta k_1k_{2,N-n}\sqrt{\big(\frac{b_{2,N-n}}{k_{2,N-n}}\big)^2+\big(\frac{b_1}{k_{1,n}}\big)^2-\delta^2}}{b_1^2-k_{1,n}^2\delta^2}\\
&\leq \frac{-b_1b_{2,N-n}}{b_1^2-k_{1,n}^2\delta^2} = \frac{-b_{2,N-n}}{b_1(1-\frac{k_{1,n}^2\delta^2}{b_1^2})} \leq \frac{-b_{2,N-n}}{b_1}.
\end{align}
Next we prove that $\frac{-b_{2,N-n}}{b_1}\leq \bar{R}_{2,N-n}.$ Indeed 
\begin{align}
\bar{R}_{2,N-n}&=\frac{-b_1b_{2,N-n}+\delta k_{1,n}k_{2,N-n}\sqrt{\big(\frac{b_{2,N-n}}{k_{2,N-n}}\big)^2+\big(\frac{b_1}{k_{1,n}}\big)^2-\delta^2}}{b_1^2-k_{1,n}^2\delta^2}\\
&=\frac{-b_{2,N-n}/b_1}{(1-\frac{k_{1,n}^2\delta^2}{b_1^2})}\left\{1-\frac{k_{1,n}k_{2,N-n}\delta}{b_1b_{2,N-n}}\sqrt{\big(\frac{b_{2,N-n}}{k_{2,N-n}}\big)^2+\big(\frac{b_1}{k_{1,n}}\big)^2-\delta^2}  \right\}\\
&=\frac{-b_{2,N-n}/b_1}{(1-\frac{k_{1,n}\delta}{b_1})(1+\frac{k_{1,n}\delta}{b_1})}\left\{1-\frac{k_{1,n}\delta}{b_1}\sqrt{1+\big(\frac{k_{2,N-n}}{b_{2,N-n}}\big)^2\Big[\big(\frac{b_1}{k_{1,n}}\big)^2-\delta^2\Big]}  \right\}\\
&\geq \frac{-b_{2,N-n}/b_1}{(1+\frac{k_{1,n}\delta}{b_1})} \geq \frac{-b_{2,N-n}}{b_1}.
\end{align}
So, we proved that $$b_{2,N-n}>0\implies \bar{R}_{1,N-n}\leq\frac{-b_{2,N-n}}{b_1}\leq \bar{R}_{2,N-n}.$$\\
\textbf{CASE II:} $b_{2,N-n}<0:$\\
For sake of simplicity, we write $-b_{2,N-n}=\hat{b}_{2,N-n}.$
Thus, 
\begin{align}
\bar{R}_{1,N-n}&=\frac{b_1\hat{b}_{2,N-n}-\delta k_1k_{2,N-n}\sqrt{\big(\frac{\hat{b}_{2,N-n}}{k_{2,N-n}}\big)^2+\big(\frac{b_1}{k_{1,n}}\big)^2-\delta^2}}{b_1^2-k_{1,n}^2\delta^2}\\
&=\frac{\hat{b}_{2,N-n}/b_1}{(1-\frac{k_{1,n}^2\delta^2}{b_1^2})}\left\{1-\frac{k_{1,n}k_{2,N-n}\delta}{b_1\hat{b}_{2,N-n}}\sqrt{\big(\frac{\hat{b}_{2,N-n}}{k_{2,N-n}}\big)^2+\big(\frac{b_1}{k_{1,n}}\big)^2-\delta^2}  \right\}\\
&=\frac{\hat{b}_{2,N-n}/b_1}{(1-\frac{k_{1,n}\delta}{b_1})(1+\frac{k_{1,n}\delta}{b_1})}\left\{1-\frac{k_{1,n}\delta}{b_1}\sqrt{1+\big(\frac{k_{2,N-n}}{\hat{b}_{2,N-n}}\big)^2\Big[\big(\frac{b_1}{k_{1,n}}\big)^2-\delta^2\Big]}  \right\}\\
&\leq \frac{\hat{b}_{2,N-n}/b_1}{(1+\frac{k_{1,n}\delta}{b_1})} \leq \frac{\hat{b}_{2,N-n}}{b_1}.
\end{align}
Next, 
\begin{align}
\bar{R}_{2,N-n}&=\frac{b_1\hat{b}_{2,N-n}+\delta k_{1,n}k_{2,N-n}\sqrt{\big(\frac{\hat{b}_{2,N-n}}{k_{2,N-n}}\big)^2+\big(\frac{b_1}{k_{1,n}}\big)^2-\delta^2}}{b_1^2-k_{1,n}^2\delta^2}\\
&=\frac{\hat{b}_{2,N-n}/b_1}{(1-\frac{k_{1,n}^2\delta^2}{b_1^2})}\left\{1+\frac{k_{1,n}k_{2,N-n}\delta}{b_1\hat{b}_{2,N-n}}\sqrt{\big(\frac{\hat{b}_{2,N-n}}{k_{2,N-n}}\big)^2+\big(\frac{b_1}{k_{1,n}}\big)^2-\delta^2}  \right\}\\
&=\frac{\hat{b}_{2,N-n}/b_1}{(1-\frac{k_{1,n}^2\delta^2}{b_1^2})}\left\{1+\frac{k_{1,n}\delta}{b_1}\sqrt{1+\big(\frac{k_{2,N-n}}{\hat{b}_{2,N-n}}\big)^2[\big(\frac{b_1}{k_{1,n}}\big)^2-\delta^2]}  \right\}\\
&\geq \frac{\hat{b}_{2,N-n}}{b_1}
\end{align}
Thus, we proved that for both cases $$\bar{R}_{1,N-n}\leq\frac{-b_{2,N-n}}{b_1}\leq \bar{R}_{2,N-n}.$$

\end{proof}

\section{Proof of theorem \ref{M2}}
\begin{proof}
\textbf{Objective Function:}
Objective function in \ref{preCT1}, in light of the deterministic trading strategies considered, is simplified as follows
\begin{equation}
\begin{aligned}
\underset{\pi_{t_{0}},\pi_{t_{1}},...\pi_{t_{N-1}}}{\text{maximize}}
& E[-e^{-\gamma(\sum_{i=0}^{N-1} \pi_i(\mu_1h + \sigma_1\sqrt{h}\epsilon_i^s) )}] \\
\implies \underset{\pi_{t_{0}},\pi_{t_{1}},...\pi_{t_{N-1}}}{\text{minimize}}
& \log[e^{-\gamma(\sum_{i=0}^{N-1} \pi_i\mu_1h )}E[e^{-\gamma(\sum_{i=0}^{N-1} \pi_i\sigma_1\sqrt{h}\epsilon_i^s) }]] \\
\implies \underset{\pi_{t_{0}},\pi_{t_{1}},...\pi_{t_{N-1}}}{\text{minimize}}
& -\gamma(\mu_1h + \sigma_1\sqrt{h})\sumIN \pi_i +\sumIN \log [(1+ e^{2\gamma(\pi_i\sigma_1\sqrt{h}) }) ] \\
\end{aligned}
\end{equation}

\textbf{Constraints :}
The constraint on trading strategy is simplified as follows 
\begin{equation}
\begin{aligned}
& \frac{b_1\sum_{i=0}^{N-1}\pi_i + b_2}{\sqrt{k_1^2\sum_{i=0}^{N-1}\pi_i^2 + k_2^2}} \leq -\delta \\
\implies & b_1^2(\sum_{i=0}^{N-1}\pi_i)^2 -\delta^2k_1^2\sum_{i=0}^{N-1}\pi_i^2 + 2b_2b_1\sumIN \pi_i + b_2^2 -k_2^2\delta^2 \geq 0 \text{ and }\sum_{i=0}^{N-1}\pi_i \le \frac{-b_2}{b_1}\\
\implies & (\sum_{i=0}^{N-1}\pi_i)^2 -\frac{ \delta^2k_1^2}{b_1^2}\sum_{i=0}^{N-1}\pi_i^2 + 2\frac{b_2}{b_1}\sumIN \pi_i + \frac{b_2^2 -k_2^2\delta^2}{b_1^2} \geq 0 \text{ and }\sum_{i=0}^{N-1}\pi_i \le \frac{-b_2}{b_1}\\
\implies & (\sum_{i=0}^{N-1}\pi_i)^2 -\frac{ \delta^2k_1^2}{b_1^2}\sum_{i=0}^{N-1}\pi_i^2 + 2\frac{b_2}{b_1}(\sumIN \pi_i +\frac{b_2}{b_1})  - \frac{b_2^2 + k_2^2\delta^2}{b_1^2} \geq 0 \text{ and }\sum_{i=0}^{N-1}\pi_i \le \frac{-b_2}{b_1}\\
\implies & \Pi'A\Pi + 2\frac{b_2}{b_1}(\textbf{1}'\Pi +\frac{b_2}{b_1}) \geq \frac{b_2^2 + k_2^2\delta^2}{b_1^2} \text{ and } \textbf{1}'\Pi \le \frac{-b_2}{b_1}\\
\implies & \Pi'A\Pi + 2c(\textbf{1}'\Pi +c) \geq c^2 + h^2 \text{ and } (\textbf{1}'\Pi+c) \leq 0\\
\end{aligned}
\label{constTrans}
\end{equation}
where $\Pi' = [\pi_{t_{0}},\pi_{t_{1}},...\pi_{t_{N-1}}]$, $c= \frac{b_2}{b_1}$, $h^2=\frac{k_2^2\delta^2}{b_1^2}$ and  $A$ is a symmetrix matrix given by, 
\[
    a_{i,j}= 
\begin{cases}
    1-\frac{\delta^2k_1^2}{b_1^2},& \text{if } i=j\\
    1,              & \text{otherwise}
\end{cases}=
\begin{cases}
    a,& \text{if } i=j\\
    1,              & \text{otherwise}
\end{cases}
\]
The characteristic equation of A is given by $$(\lambda - (a-1))^{n-1}(\lambda - (a+n-1))=0,$$ thus $A$ has $n-1$ negative eigenvalues and one positive eigenvalue. Let us first visualize the solution in two dimensions, which will then extend to N dimensions.\\ 

\textbf{Two Dimensions:}\\
In two dimensions, the eigen values are $a+1 (>0)$ and $a-1(<0)$ and  the constraint on the trading strategies reduces to 
\begin{equation}
\begin{aligned}
a(\pi_1^2 + \pi_2^2) + 2\pi_1\pi_2 + 2c(\pi_1+\pi_2) + c^2-h^2 \geq 0 \text{ and } \pi_1+\pi_2 \leq -c.
\end{aligned}
\end{equation}
Since the objective function and the domain is convex, the first order conditions (FOC) lead to $\pi_1=\pi_2.$  \\

\textbf{N-Dimensions:}\\
Since $A$ has $N-1$ negative eigen values and one positive eigenvalue, the constraint is a 2-sheeted, N-dimensional hyperboloid. Moreover, similar to 2 dimensions, the constraint $\textbf{1}'\Pi \le \frac{-b_2}{b_1}$ also passes through the two cones of the hyperboloid, leading to a convex constraint for the domain. FOC lead to $\pi_1=...=\pi_n.$ 
Let $\Pi = X + b\textbf{1}$, where $b=\frac{-c}{a+n-1}$. This transformation simplifies the constraints as- 
\begin{equation}
\begin{aligned}
& \Pi'A\Pi + 2c(\textbf{1}'\Pi +c) \geq c^2 + h^2 \\
\implies & (X + b\textbf{1})'A(X + b\textbf{1}) + 2c(\textbf{1}'(X + b\textbf{1}) +c) \geq c^2 + h^2 \\
\implies & X'AX + b^2\bOne'A\bOne +2cb\bOne'\bOne + c^2-h^2 \geq 0 \\
\implies & X'AX -\frac{nc^2}{a+n-1} +c^2-h^2 \geq 0 \\
\implies & X'AX -[ \frac{c^2 (1-a)}{a+n-1} + h^2]  \geq 0. \\
\end{aligned}
\end{equation}

Furthermore, the linear constraint gets transformed as - 
\begin{equation}
\begin{aligned}
& \bOne'\Pi+c \leq 0 \\
\implies & \bOne'(X+b\bOne)+c \leq 0 \\
\implies & \bOne'X \leq -c- nb \\
\implies & \bOne'X \leq \frac{c(1-a)}{a+n-1}. \\
\end{aligned}
\end{equation}
Since we have established that the solution exists when, $\pi_1=...=\pi_n \implies x_1=...=x_n.$ Thus, the solution of the deterministic precomittment strategy is
\begin{equation}
\hat{\pi}_{t_i} = \min\lbrace  \bar{\pi},-b -\sqrt{\frac{c^2(1-a)}{n(a+n-1)^2} + \frac{h^2}{{n(a+n-1)}}} \rbrace \text{ }\forall i \in \{0,1,...,N-1 \}
\end{equation}

\end{proof}

\section{Proof of Theorem \ref{unm} }
\begin{proof}
The proof of the result follows from Theorem \ref{un} and backward induction.
\end{proof}

\bibliography{references}


\end{document}